\documentclass[12pt]{article}
\usepackage{amsmath,amssymb}
\usepackage[T1]{fontenc}
\usepackage{latexsym}
\usepackage{epsfig}
\usepackage{psfrag}
\usepackage[utf8]{inputenc}
\usepackage[a4paper, top=3cm, bottom=3cm ]{geometry}
\newtheorem{theorem}{Theorem}[section]

\newtheorem{lemma}{Lemma}[section]

\newtheorem{pr}{Proposition}[section]

\title
{Explicit formulas of the Bergman kernel for some Reinhardt domains}

\author{\normalsize Tomasz Beberok \\
\small Faculty of Mathematics and Computer Science, Jagiellonian University,\\
\small Lojasiewicza 6, 30-048 Krakow, Poland \\}

\date{}

\begin{document}

\begin{center}
  \textbf{Explicit formulas of the Bergman kernel for some Reinhardt domains}
\end{center}
\vskip1em
\begin{center}
  Tomasz Beberok
\end{center}

\vskip3em

In this paper we obtain the closed forms of some hypergeometric functions. As an application, we obtain the explicit forms of the Bergman kernel functions for Reinhardt domains $\{|z_3|^{\lambda} < |z_1|^{2p}  + |z_2|^2, \quad |z_1|^{2p}  + |z_2|^2 < |z_1|^{p} \}$  and \\ $\{|z_4|^{\lambda} < (|z_1|^2 + |z_2|^2)^{p}  + |z_3|^2, \quad (|z_1|^2 + |z_2|^2)^{p}  + |z_3|^2 < (|z_1|^2 + |z_2|^2 )^{p/2} \}$.
\vskip1em


\section{\bf Introduction}

In 1921, S. Bergman introduced a kernel function, which is now known as the Bergman kernel function. It is well known that there exists a unique Bergman kernel function for each bounded domain in $\mathbb{C}^n$. Computation of the Bergman kernel function by explicit formulas is an important research direction in several
complex variables. For which domains can the Bergman kernel function be computed by explicit formulas?  Many mathematicians (\cite{DA1},  \cite{FH}, \cite{F2}, \cite{F3}, \cite{Park1}) have made efforts to find the explicit formulas of the Bergman kernel for nonhomogeneous domains. Consider the complex ellipsoids or egg domains $\Omega_{p} := \{z \in \mathbb{C}^n : \sum\limits_{j=1}^{n} |z_j |^{2p_j} <1\}$, where $p = (p_1, . . . , p_n)$ for $p_j > 0$. The precise growth estimate of the Bergman kernel near a boundary point on the complex ellipsoid was studied in \cite{ZY}. However, it is not easy to get the closed forms of the Bergman kernel for $D_p$.
In the case when $p_1, . . . , p_n$ are reciprocals of positive integers, Zinov'ev \cite{Z} computed the Bergman kernel for $D_p$ explicitly. What happens if each $p_j$ is a positive integer? The known case is when $p = (1, . . . , 1, p_n), p_n > 0$, for which  J. P. D’Angelo \cite{DA1, DA2} obtained the Bergman kernel. J.-D. Park computed the Bergman kernel for  $p =(2,2)$ and $p=(2,2,2)$ in \cite{Park1} and \cite{Park2} respectively. The goal of this paper is to give explicit formula for the domains $\{|z_3|^{\lambda} < |z_1|^{2p}  + |z_2|^2, \ |z_1|^{2p}  + |z_2|^2 < |z_1|^{p} \}$  and $\{|z_4|^{\lambda} < (|z_1|^2 + |z_2|^2)^{p}  + |z_3|^2, \ (|z_1|^2 + |z_2|^2)^{p}  + |z_3|^2 < (|z_1|^2 + |z_2|^2)^{p/2} \}$.

\begin{theorem}
For any positive real numbers $\lambda, p$ the Bergman kernel for the domain  $$D_1=\{z \in  \mathbb{C}^4 \colon |z_4|^{\lambda} < (|z_1|^2 + |z_2|^2)^{p}  + |z_3|^2, \ (|z_1|^2 + |z_2|^2)^{p}  + |z_3|^2 < (|z_1|^2 + |z_2|^2)^{\frac{p}{2}}\}$$ is given by

\begin{align*}
K_{D_1}((z_1,z_2,z_3,z_4),(\zeta_1,\zeta_2,\zeta_3, \zeta_4))= D_{p,\lambda, \alpha} \left\{  C \frac { ( \frac{2}{\lambda} -1 )(\sqrt{1-4\nu_3}-1)  + \frac{8\nu_3}{\lambda} }{(1-\mu_4)^2 (1- \mu_1 - \mu_2)^2 }  \right. \\  \left.  +  C  \frac{ 4 \sqrt{1-4\nu_3}  + 16\nu_3 - 4} {p (1-\mu_4)^2 (1- \mu_1 - \mu_2)^3}    +  C   \frac{4\mu_4 \left(\sqrt{1-4\nu_3}  + 4\nu_3 - 1\right) }{\lambda (1-\mu_4)^3 (1- \mu_1 - \mu_2)^2} \right\} ,
\end{align*}
where
$$ \nu_1 =    z_1 \overline{\zeta_1}, \quad  \nu_2 =    z_2 \overline{\zeta_2}, \quad \nu_3 =    z_3 \overline{\zeta_3},  \quad  \nu_4 =    z_4 \overline{\zeta_4}, $$
$$D_{p,\lambda, \alpha}f = \frac{p}{ \pi^4} \left(\frac{2}{p}  \frac{\partial}{\partial \nu_1} \nu_1 + \frac{2}{p}   \frac{\partial}{\partial \nu_2} \nu_2 +  \frac{\partial}{\partial \nu_3} \nu_3 + \frac{2}{\lambda}  \frac{\partial}{\partial \nu_4} \nu_4 \right)f ,$$
$$\mu_1=  \frac{2^{2/p} \nu_1} {\left(1+ \sqrt{1- 4\nu_3}\right)^{2/p} }, \quad \mu_2=  \frac{2^{2/p} \nu_2} {\left(1+ \sqrt{1- 4\nu_3}\right)^{2/p} } , \quad  \mu_4=  \frac{2^{2/\lambda} \nu_4} {\left(1+ \sqrt{1- 4\nu_3}\right)^{2/\lambda} }$$
and
$$ C=\frac{ 2^{4/p + 2/\lambda} }{4\nu_3 (1-4\nu_3)^{3/2} (1+ \sqrt{1- 4\nu_3})^{4/p + 2/\lambda -1}}.$$

\end{theorem}

\begin{theorem}
 The Bergman kernel for  $$D_2=\{z \in  \mathbb{C}^3 \colon |z_3|^{2} < |z_1|^{4}  + |z_2|^2, \quad |z_1|^{4}  + |z_2|^2 < |z_1|^{2} \}$$ is given by

\begin{align*}
K_{D_2}((z_1,z_2,z_3),(\zeta_1,\zeta_2,\zeta_3))= \frac{2\nu_1^4 -( \nu_1^2 \nu_3 + \nu_1^3)(\nu_1^2 + \nu_2)}{ \pi^3 (\nu_1 - \nu_3)^3 (\nu_1 - \nu_1^2 - \nu_2)^3} ,
\end{align*}
where $ \nu_1 =    z_1 \overline{\zeta_1}, \quad  \nu_2 =    z_2 \overline{\zeta_2}, \quad \nu_3 =    z_3 \overline{\zeta_3}$.

\end{theorem}

In 1995 Francsics and Hanges \cite{FH} expressed the Bergman kernel for complex ellipsoids $\Omega_{p_1,...,p_n}$ in terms of Appell’s multivariable hypergeometric functions which are still infinite series. Recall that an Appell’s hypergeometric function \cite{AK} is defined by $$F^{(n)}_A (\alpha; \beta; \gamma; \zeta)=\sum_{m_1=0}^{\infty} \cdots \sum_{m_n=0}^{\infty} \frac{(\alpha)_{m_1+\ldots+m_n} (\beta_1)_{m_1} \cdots (\beta_n)_{m_n} }{m_1! \cdots m_n! (\gamma_1)_{m_1} \cdots (\gamma_n)_{m_n} } \zeta_1^{m_1} \cdots \zeta_n^{m_n}, $$

\noindent where $(a)_m = \Gamma(a+m)/\Gamma (a)$. In particular we write $F_2 = F^{(2)}_A$ and $F = F^{(1)}_A$. In fact, the Bergman kernel $K(z,w)$ for $\Omega_{p_1,...,p_n}$ is given in \cite{FH} by

\begin{eqnarray*}
K(z,w)=\frac{\prod\limits_{j=1}^n p_j}{\pi^n} \sum\limits_{k_1=0}^{p_1 -1} \cdots \sum\limits_{k_n=0}^{p_n - 1} \frac{\Gamma\left(1+ \sum\limits_{j=1}^n \frac{k_j +1}{p_j}\right)}{\prod\limits_{j=1}^n \Gamma\left(\frac{k_j +1}{p_j}\right)} (z\overline{w})^k  \\ \times F^{(n)}_A \left(1+ \sum\limits_{j=1}^n \frac{k_j +1}{p_j};\mathbf{1};\frac{\mathbf{k+1}}{\mathbf{p}};(z\overline{w})^p\right),
\end{eqnarray*}

\noindent where $(z\overline{w})^k = (z_1\overline{w_1})^{k_1}\cdots(z_n \overline{w_n})^{k_n} $. Here we following by Park used the notation $\mathbf{1} = \underbrace{(1,\ldots,1)}_n$ and $\frac{\mathbf{k+1}}{\mathbf{p}}=\underbrace{\left(\frac{k_1 +1}{p_1},\ldots, \frac{k_n +1}{p_n} \right)}_n.$

\section{\bf  Explicit formulas of hypergeometric functions}

In this section we will express the sum of the series $\sum\limits_{m=0}^{\infty} \frac{(a)_{2m_1+\ldots + 2m_n} x_1^{m_1} \cdots x_n^{m_n} }{(c)_{m_1+\ldots + m_n} m_1! \cdots m_n! }$ in terms of Gauss hypergeometric function.

\begin{lemma}\label{lem3}

\noindent For $|x_1| + \ldots + |x_r| <1/4$, we have

\begin{align*}
 F \left(\frac{a}{2},\frac{a+1}{2} ; c ; 4(x_1+ \ldots + x_r ) \right)=  \sum_{m=0}^{\infty} \frac{(a)_{2m_1+\ldots + 2m_n} x_1^{m_1} \cdots x_n^{m_n} }{(c)_{m_1+\ldots + m_n} m_1! \cdots m_n! }
\end{align*}

\end{lemma}

\begin{proof}
Using well known rules for Pochhammer symbol $(2z)_{2k}=4^k (z)_k (z+1/2)_k$ we have
$$\sum_{m=0}^{\infty} \frac{(a)_{2m_1+\ldots + 2m_n} x_1^{m_1} \cdots x_n^{m_n} }{(c)_{m_1+\ldots + m_n} m_1! \cdots m_n! }= \sum_{m=0}^{\infty} \frac{4^{|m|} (a/2)_{|m|} (a/2+1/2)_{|m|}}{(c)_{|m|} m_1! \cdots m_n!} x_1^{m_1} \cdots x_n^{m_n}  .$$

Now using $(z)_{n+k}=(z)_n (z+n)_k$ and sum out of $x_1$ variable we have

\begin{eqnarray*}
\sum_{m=0}^{\infty} \frac{4^{|m|} (a/2)_{|m|} (a/2+1/2)_{|m|}}{(c)_{|m|} m_1! \cdots m_n!} x_1^{m_1} \cdots x_n^{m_n} =\\ \sum_{m_2,\ldots,m_n=0}^{\infty}  \frac{4^{m_2 + \ldots + m_n} (a/2)_{m_2 + \ldots + m_n} (a/2+1/2)_{m_2 + \ldots + m_n}}{(c)_{m_2 + \ldots + m_n} m_2! \cdots m_n!} x_2^{m_2} \cdots x_n^{m_n} \\ \times F\left(\frac{a}{2}+ m_2 + \ldots + m_n , \frac{a+1}{2} + m_2 + \ldots + m_n; c+ m_2 + \ldots + m_n; 4x_1 \right).
\end{eqnarray*}

In the other hand by decomposition formulas for the Appell function $F^{(r)}_A$  in $r$ ($r>1$) variables we have (see for more details \cite{HS})

\begin{align*}
&F^{(r)}_A (a,b_1,\ldots,b_r ; c_1,\ldots, c_r; y_1,\ldots,y_r )=  \\ &\sum_{m_2,\ldots,m_r=0}^{\infty}  \frac{ (a)_{m_2 + \ldots + m_r} (b_1)_{m_2 + \ldots + m_r} (b_2)_{m_2} \cdots (b_r)_{m_r} }{m_2! \cdots m_r!  (c_1)_{m_2 + \ldots + m_r} (c_2)_{m_2} \cdots (c_r)_{m_r} } y_1^{m_2 + \ldots + m_r} y_2^{m_2} \cdots y_r^{m_r} \\ & \cdot F\left(a + m_2 + \ldots + m_r , b_1 + m_2 + \ldots + m_r; c_1+ m_2 + \ldots + m_r; y_1 \right)\\ & \cdot F^{(r-1)}_A( a + m_2 + \ldots + m_r, b_2+m_2, \ldots , b_r+ m_r; c_2 + m_2, \ldots, c_r + m_r; y_2,\ldots,y_r)
\end{align*}
Next we set $b_i=c_i$ for $i=2,\ldots,r$. After doing so, and using well know formula $F^{(s)}_A (a,b_1,\ldots,b_s; b_1,\ldots,b_s;z_1,\ldots,z_s) = \frac{1}{(1-z_1-\ldots - z_s)^a} $ we obtain

\begin{align*}
&F^{(r)}_A (a,b_1,b_2,\ldots,b_r ; c_1,b_2\ldots, b_r; y_1,\ldots,y_r )=  \\ &\sum_{m_2,\ldots,m_r=0}^{\infty}  \frac{ (a)_{m_2 + \ldots + m_r} (b_1)_{m_2 + \ldots + m_r}  }{m_2! \cdots m_r!  (c_1)_{m_2 + \ldots + m_r} } y_1^{m_2 + \ldots + m_r} y_2^{m_2} \cdots y_r^{m_r} \\ & \cdot F\left(a + m_2 + \ldots + m_r , b_1 + m_2 + \ldots + m_r; c_1+ m_2 + \ldots + m_r; y_1 \right)\\& \cdot  \frac{1}{(1-y_2-\ldots -y_r)^{a+m_2 + \ldots + m_r}}
\end{align*}

Finally putting $y_1=4 x_1$ and $y_i = \frac{ x_i}{|x|}$ for $i=2, \ldots, r$, where $|x|=x_1+\ldots + x_r$ and using well know formula $F^{(s)}_A(a,b_1,\ldots,b_i,\ldots,b_s; c_1,\ldots,b_i,\ldots,c_s;z_1,\ldots,z_s)= (1-z_i)^{-a} F^{(s-1)}_A(a,b_1,\ldots,b_{i-1},b_{i+1},\ldots,b_s; c_1,\ldots,c_{i-1},c_{i+1},\ldots,c_s;\frac{z_1}{1-z_i},\ldots,\frac{z_s}{1-z_i}) $  we obtain the desired result.
\end{proof}

The following lemma will be useful to explicit computation of Bergman kernel function for the domain  $$\Omega= \left\{ (z_1,z_2,z_3) \in \mathbb{C}^{n+m+k}  \colon  \|z_1\|^{\lambda} < \|z_2\|^{2p} + \|z_3\|^2, \ \|z_2\|^{2p} + \|z_3\|^2  < \|z_2\|^{p} \right\}$$ for $k=2$ and $k=3$.

\begin{lemma}\label{lem2f1}

   \begin{align*}
      &F\left( \frac{3+a}{2}, \frac{4+a}{2} ; a ; z \right)= \frac{  \left(-a-1 + \left(a - \frac{1}{2}\right )z \right) \left(\left(a - \frac{5}{2} \right)z - a\right) - \left(\frac{1-a}{2}\right) \left(\frac{2-a}{2}\right)z   }{ a (a+1) (a+2) z (z-1) (1-z)^{3/2} ( \sqrt{1-z} +1 )^a } \\ & \cdot 2^{a+1} (a (\sqrt{1-z} -1 ) + (a+1)z) - \frac{2^a \left(\left(a - \frac{5}{2} \right)z - a\right) \left(a+a^2\right)z   }{a (a+1) (a+2) (z-1)^2 \sqrt{1-z} (\sqrt{1-z}+1)^{a+2} }
   \end{align*}

\begin{align*}
      F\left( \frac{2+a}{2}, \frac{3+a}{2} ; a ; z \right)=& \frac{ 2^a (a-a^2)z  }{2 a (a+1)(z-1) \sqrt{1-z} (\sqrt{1-z} +1)^{a+1}  } \\ &+ \frac{ (-3/2 z -a) 2^a ((a-1) (\sqrt{1-z} - 1 )+ az)  }{ a (a+1) z (z-1) (1-z)^{3/2} (\sqrt{1-z} +1)^{a-1} }
   \end{align*}

\end{lemma}

\begin{proof}
In order to prove the above lemma, we need the following well-known formulas:

$$ F\left( \frac{a+3}{2}, \frac{a+4}{2}; a+2; z \right) = \frac{2^{a+1}(a (\sqrt{1-z}-1) + (a+1)z) }{(a+2)z (1-z)^{3/2} ( \sqrt{1-z} +1)^a } \quad \text{and}$$

$$ F\left( \frac{a+3}{2}, \frac{a+4}{2}; a+3; z \right) = \frac{2^{a+2}}{ \sqrt{1-z} ( \sqrt{1-z} +1 )^{a+2} }$$

Now lemma \ref{lem2f1} follows from recurrence identity for Gauss hypergeometric function

\begin{align*}
& F\left( \frac{a+3}{2}, \frac{a+4}{2}; a; z \right)= C_2 F\left( \frac{a+3}{2}, \frac{a+4}{2}; a+2; z \right) \\& - \frac{\left(a+a^2 \right) z   }{ 4(a+1)(a+2)(z-1) }   \cdot C_1 F\left( \frac{a+3}{2}, \frac{a+4}{2}; a+3; z \right),
\end{align*}
where $C_2= \frac{ \left(-a-1 + \left(a - \frac{1}{2}\right )z \right) \left(\left(a - \frac{5}{2} \right)z - a\right) - \left(\frac{1-a}{2}\right) \left(\frac{2-a}{2}\right)z }{  a (a+1)  (z-1)}$ and $C_1= \frac{\left(a - \frac{5}{2} \right)z - a}{a(z-1)}$.

\end{proof}

\section{\bf Computation of the kernel.}

Let $\Omega$ be a bounded domain in $\mathbb{C}^N$. The Bergman projection operator is the orthogonal projection $P$ from $L^2(\Omega)$ to the closed subspace of holomorphic square integrable functions. The Bergman kernel function is the integral kernel associated with the Bergman projection $P$. The operator $P$ and the function $K$ are therefore related by $$Pf(\zeta)=\int_{\Omega} K(\zeta,\xi) f (\xi) dV(\xi). $$
 It is well known that $K$ can be expressed by summation of an orthonormal series. More precisely, suppose that $\{ \Phi_{\alpha}\}$ from a complete orthonormal set for the Hilbert space of holomorphic functions in $L^2(\Omega)$. Then we have $$K(\zeta,\xi)= \sum_{\alpha} \Phi_{\alpha}(\zeta)  \overline{\Phi_{\alpha}(\xi)} .$$ Let $\zeta=(z_1,z_2,z_3,z_4) \in \mathbb{C}^4.$ Put $\Phi_{\alpha}(\zeta)= z_1^{\alpha_1} z_2^{\alpha_2} z_3^{\alpha_3} z_4^{\alpha_4}$. It is well known, that function $f$ holomorphic in a Reinhardt domain $D \subset \mathbb{C}^n$  has a “global” expansion into a Laurent series $f(z)=\sum_{\alpha \in \mathbb{Z}^n} a_{\alpha} z^{\alpha}$, $z \in D$ (see Proposition 1.7.15 (c) in \cite{JP}). Moreover if $D \cap (\mathbb{C}^{j-1} \times \{0\} \times \mathbb{C}^{n-j} ) \neq \emptyset $, $j=1,\ldots, n$ then $a_{\alpha}=0$ for $\alpha \in \mathbb{Z}^n \setminus \mathbb{Z}^n_{+}$ (see Proposition 1.6.5 (c) in \cite{JP}). Therefore  $\{\Phi_{\alpha} \}$ such that each $\alpha_i \geq 0$ is a complete orthogonal set for $L^2(D_1)$. \newline If $D$ is a Reinhardt domain, $f \in L^2_a(D) := \mathcal{O}(D) \cap L^2(D)$,  $f(z)=\sum_{\alpha \in \mathbb{Z}^n} a_{\alpha} z^{\alpha}$, then $\{  z^{\alpha} \colon \alpha \in \sum(f)  \} \subset L^2_a(D), $ where $\sum(f):=\{ \alpha \in \mathbb{Z}^n \colon a_{\alpha} \neq 0 \}$ (for proof see \cite{JP} p. 67). Thus it is easy to check, that the set $\{z_1^{\alpha_1} z_2^{\alpha_2} z_3^{\alpha_3}\colon   \alpha_2 \geq 0,  \alpha_3 \geq 0, \alpha_1 \geq -2 - \alpha_2 - \alpha_3\}$ is a complete orthogonal set for $L^2_a(D_2)$.

 \begin{pr}
 The squared $L^2(D_2)$-norms satisfy
 \begin{equation}
 \|z^{\alpha} \|^2_{L^2} = \frac{  \pi^3   \Gamma\left( \alpha_2 +1 \right)  \Gamma\left( \alpha_1 + \alpha_2 + \alpha_3 + 3 \right)}{ ( \alpha_3 +1) \left( \alpha_1 + 2\alpha_2  +  2\alpha_3 +5 \right)   \Gamma\left( \alpha_1 + 2\alpha_2 + \alpha_3 + 4 \right)  },
 \end{equation}
where $\alpha_2 \geq 0,  \alpha_3 \geq 0, \alpha_1 \geq -2 - \alpha_2 - \alpha_3$.
 \end{pr}

\begin{proof}
$$\|z^{\alpha} \|^2_{L^2} = \int\limits_{D_2} |z|^{2\alpha} dV(z)$$

we introduce polar coordinate in each variable by putting $z_j=r_j e^{i\varphi_j}$, for $j=1,2,3$. After doing so, and integrating out the angular variables we have

$$(2 \pi)^3 \int_{Re(D_2)} r^{2 \alpha +1} r_2^{2 \alpha_2+1} r_3^{2 \alpha_3 +1}dV(r),$$

where $Re(D_2)=\{ r \in \mathbb{R}_{+}^3 : r_3^{2} < r_1^{4} + r_2^2, \ r_1^{4} + r_2^2  < r_1^{2} \}$. Next we set $r_1^2=t$ and change variables again. We obtain

$$ \frac{(2 \pi)^3}{2} \int_{Re(D_2')} t^{ \alpha_1} r_2^{ 2\alpha_2 +1} r_3^{ 2\alpha_3 +1}  dt dr_2 dr_3,$$

where $Re(D_2')=\{ (t,r_2,r_3) \in \mathbb{R}_{+}^3 : r_3^{2} < t^{2} + r_2^2, \ t^{2} + r_2^2  < t \}$. Next we use spherical coordinate in the $t,r_2$ variables to obtain

$$ 4 \pi^3 \int_{0}^{\pi/2}  \int_0^{\sin \theta } \int_0^{\rho}  \rho^{ \alpha_1 + 2\alpha_2  + 2 }  (\cos \theta )^{2\alpha_2 + 1} (\sin \theta)^{ \alpha_1}  r_3^{2 \alpha_3 +1}  dr_3 d\rho  d \theta$$
After integrating out $r_3, \rho$ and $\theta$ we obtain the desired result.
\end{proof}

 \begin{pr}
 The squared $L^2(D_1)$-norms satisfy
 \begin{equation}
 \|z^{\alpha} \|^2_{L^2} = \frac{  \pi^4  \Gamma\left( \alpha_1 +1 \right) \Gamma\left( \alpha_2 +1 \right) \Gamma\left( \alpha_3 + 1 \right)  \Gamma\left( \frac{2\alpha_1 + 2\alpha_2 + 4}{p}  + \alpha_3 + \frac{2\alpha_4 +2}{ \lambda} + 1 \right)}{ p \Gamma\left( \alpha_1 + \alpha_2 + 2 \right) ( \alpha_4 +1) \left( s + \frac{\alpha_4 + 1}{ \lambda} \right)   \Gamma\left(2 s \right)  },
 \end{equation}
where $s= \frac{\alpha_1 +  \alpha_2 + 2}{p}  + \alpha_3 + \frac{\alpha_4 + 1}{ \lambda} + 1 $.
 \end{pr}

\begin{proof}
$$\|z^{\alpha} \|^2_{L^2} = \int\limits_{D_1} |z|^{2\alpha} dV(z)$$

we introduce polar coordinate in each variable by putting $z_j=r_j e^{i\varphi_j}$, for $j=1,2,3,4$. After doing so, and integrating out the angular variables we have

$$(2 \pi)^4 \int_{Re(D_1)} r^{2 \alpha +1}\, dV(r),$$

where $Re(D_1)=\{ r \in \mathbb{R}_{+}^4 : r_4^{\lambda} < (r_1^2 + r_2^2)^{p} +  r_3^2, \ (r_1^2 + r_2^2)^p + r_3^2 < r_1^{p/2} \}$. Next we set $r_1=\rho \cos \omega$, $r_2=\rho \sin \omega$ and change variables again. We obtain

$$(2 \pi)^4 \int\limits_{\substack {r_4^{\lambda} < \rho^{2p} +  r_3^2 \\  \rho^{2p} + r_3^2 < \rho^{p}}} \int_0^{\pi/2}   \rho^{2\alpha_1  + 2\alpha_2  + 3 }  (\cos \omega )^{2\alpha_1 + 1} (\sin \omega )^{2\alpha_2+1} r_3^{ 2\alpha_3 +1} r_4^{ 2\alpha_4 +1}\, d\omega d\rho dr_3 dr_4, $$
integrating out of $\omega$ variable we have

$$\frac{(2 \pi)^4 \Gamma(\alpha_1 +1 )   \Gamma(\alpha_2 +1 )  }{2 \Gamma(\alpha_1 + \alpha_2 + 2 )} \int\limits_{\substack {r_4^{\lambda} < \rho^{2p} +  r_3^2 \\  \rho^{2p} + r_3^2 < \rho^{p}}}    \rho^{2\alpha_1  + 2\alpha_2  + 3 }   r_3^{ 2\alpha_3 +1} r_4^{ 2\alpha_4 +1}\, d\omega d\rho dr_3 dr_4, $$

After little calculation we obtain

$$ C \int_{0}^{\pi/2}  \int_0^{\cos \theta } \int_0^{R^{2/ \lambda}}  R^{\frac{2\alpha_1 + 2\alpha_2 + 4}{p} +  2\alpha_3 + 1 }  (\sin \theta )^{ 2\alpha_3 + 1} (\cos \theta)^{\frac{2\alpha_1 + 2\alpha_2 + 4}{p} -1}  r_4^{2 \alpha_4 +1} dr_4 dR  d \theta,$$
where $C= \frac{8 \pi^4 \Gamma(\alpha_1 +1 )   \Gamma(\alpha_2 +1 )}{p \Gamma( \alpha_1 + \alpha_2 +2) }$. After integrating out $r_4, R$ and $\theta$ we obtain the desired result.
\end{proof}

Now we will prove main theorem. We set $\nu_j=z_j \overline{w_j}$, for $j=1,2,3,4$. By the series representation Bergman kernel for $D_1$ is given by
\begin{eqnarray*}
 K(z,w)=  \frac{p}{ \pi^4} \sum_{\alpha=0}^{\infty}  \frac{  \Gamma\left(  \alpha_1 + \alpha_2 + 2 \right) ( \alpha_4 +1) \left( s + \frac{\alpha_4 + 1}{ \lambda} \right)   \Gamma\left( 2 s \right)      }{     \Gamma\left( \frac{2\alpha_1 + 2\alpha_2 + 4}{p} + \alpha_3 + \frac{2\alpha_4 +2}{ \lambda} + 1 \right) \alpha_1! \alpha_2! \alpha_3!  }  \nu^{\alpha}
 \end{eqnarray*}
If we define

\begin{eqnarray*}
 G=   \sum_{\alpha=0}^{\infty}  \frac{  \Gamma\left(  \alpha_1 + \alpha_2 + 2 \right) ( \alpha_4 +1)  \Gamma\left( 2 s \right)      }{     \Gamma\left( \frac{2\alpha_1 + 2\alpha_2 + 4}{p} + \alpha_3 + \frac{2\alpha_4 +2}{ \lambda} + 1 \right) \alpha_1! \alpha_2! \alpha_3!  }  \nu^{\alpha}
 \end{eqnarray*}

then we can write

\begin{eqnarray*}
 K(z,w)=  D_{p,\lambda, \alpha} G,
 \end{eqnarray*}

where $ D_{p,\lambda, \alpha}$ is a differential operator defined by
$$ D_{p,\lambda, \alpha}f= \frac{p}{ \pi^4} \left(\frac{2}{p}  \frac{\partial}{\partial \nu_1} \nu_1 + \frac{2}{p}   \frac{\partial}{\partial \nu_2} \nu_2 +  \frac{\partial}{\partial \nu_3} \nu_3 + \frac{2}{\lambda}  \frac{\partial}{\partial \nu_4} \nu_4 \right)f. $$

Now we sum out the  $\nu_3$ variable using lemma \ref{lem3}, we obtain

\begin{align*}
G= \sum_{\alpha_1,\alpha_2,\alpha_4=0}^{\infty}   \frac{ \Gamma\left(  \alpha_1 + \alpha_2 + 2 \right) (\alpha_4 +1) \Gamma(a+1) F\left( \frac{1+a}{2}, \frac{2+a}{2} ; a ; 4 \nu_3 \right)  }{\Gamma(a) \alpha_1! \alpha_2!   }  \nu_1^{\alpha_1} \nu_2^{\alpha_2} \nu_4^{\alpha_4} ,
\end{align*}

where $a=\frac{2\alpha_1 + 2\alpha_2 + 4}{p} + \frac{2\alpha_4 +2}{ \lambda} + 1 $.
After some calculation using $$ F\left( \frac{a+1}{2}, \frac{a+2}{2}; a; z \right) = \frac{2^{a-1}((a-2) (\sqrt{1-z}-1) + (a-1)z) }{a z (1-z)^{3/2} ( \sqrt{1-z} +1)^{a-2} } $$
and  $z\Gamma(z)=\Gamma(z+1)$, we have

$$G=C  \sum_{\alpha_1,\alpha_2,\alpha_4=0}^{\infty} \frac{ \Gamma\left(  \alpha_1 + \alpha_2 + 2 \right) (\alpha_4 +1) }{  {[(a-2) (\sqrt{1-4\nu_3}-1) + (a-1)4\nu_3]}^{-1}  \alpha_1! \alpha_2!   }  \mu_1^{\alpha_1} \mu_2^{\alpha_2} \mu_4^{\alpha_4}, $$

where $C=\frac{ 2^{4/p + 2/\lambda} }{4\nu_3 (1-4\nu_3)^{3/2} (1+ \sqrt{1- 4\nu_3})^{4/p + 2/\lambda -1}}$, $\mu_i= \frac{2^{2/p} \nu_i} {(1+ \sqrt{1- 4\nu_3})^{2/p} }$ for $i=1,2$ and $\mu_4= \frac{2^{2/\lambda} \nu_4} {(1+ \sqrt{1- 4\nu_3})^{2/\lambda} } .$

\begin{align*}
&G= C \sum_{\alpha_1,\alpha_2,\alpha_4=0}^{\infty} \frac{ \Gamma\left(  \alpha_1 + \alpha_2 + 2 \right) (\alpha_4 +1) }  {  \alpha_1! \alpha_2!   }   \\ & \cdot \left[( \frac{2}{\lambda} -1 )(\sqrt{1-4\nu_3}-1)  + \frac{8\nu_3}{\lambda}\right] \mu_1^{\alpha_1} \mu_2^{\alpha_2} \mu_4^{\alpha_4}  \\& +  C \sum_{\alpha_1,\alpha_2,\alpha_4=0}^{\infty} \frac{2 \Gamma\left(  \alpha_1 + \alpha_2 + 3 \right) (\alpha_4 +1)}  { p  \alpha_1! \alpha_2!   }    \left[\sqrt{1-4\nu_3}  + 4\nu_3 - 1\right] \mu_1^{\alpha_1} \mu_2^{\alpha_2} \mu_4^{\alpha_4}  \\& +  C \sum_{\alpha_1,\alpha_2,\alpha_4=0}^{\infty} \frac{2 \Gamma\left(  \alpha_1 + \alpha_2 + 2 \right) (\alpha_4 +1) \alpha_4 }  { \lambda  \alpha_1! \alpha_2!   }    \left[\sqrt{1-4\nu_3}  + 4\nu_3 - 1\right] \mu_1^{\alpha_1} \mu_2^{\alpha_2} \mu_4^{\alpha_4}
\end{align*}

Finally using well knows formulas $\sum\limits_{k=0}^{\infty } k (k+1) x^k= \frac{2 x}{(1-x)^3}$, $\sum\limits_{k=0}^{\infty } (k+1) x^k=\frac{1}{(1-x)^2}$ and $F_2(r,1,1;1,1,x,y)= \frac{1}{(1-x-y)^r}$, we obtain

\begin{align*}
&G= C \frac { ( \frac{2}{\lambda} -1 )(\sqrt{1-4\nu_3}-1)  + \frac{8\nu_3}{\lambda} }{(1-\mu_4)^2 (1- \mu_1 - \mu_2)^2 }   +  C  \frac{ 4 \sqrt{1-4\nu_3}  + 16\nu_3 - 4} {p (1-\mu_4)^2 (1- \mu_1 - \mu_2)^3}   \\& +  C  \left[\sqrt{1-4\nu_3}  + 4\nu_3 - 1\right] \frac{4\mu_4}{\lambda (1-\mu_4)^3 (1- \mu_1 - \mu_2)^2}
\end{align*}

Similarly we can compute Bergman kernel for $$\Omega= \left\{ (z_1,z_2,z_3) \in \mathbb{C}^{n+m+k}  \colon  \|z_1\|^{\lambda} < \|z_2\|^{2p} + \|z_3\|^2, \ \|z_2\|^{2p} + \|z_3\|^2  < \|z_2\|^{p} \right\}.$$

Now we will prove Theorem 1.2. Similarly as above Bergman kernel for $D_2$ is given by
\begin{eqnarray*} \sum_{\alpha_1, \alpha_2=0}^{\infty}    \sum_{\alpha_1=-2-\alpha_2 - \alpha_3}^{\infty} \frac{  ( \alpha_3 +1) \left( \alpha_1 + 2\alpha_2  +  2\alpha_3 +5 \right)   \Gamma\left( \alpha_1 + 2\alpha_2 + \alpha_3 + 4 \right)  }{  \pi^3  \Gamma\left( \alpha_2 +1 \right)  \Gamma\left( \alpha_1 + \alpha_2 + \alpha_3 + 3 \right) } \nu^{\alpha}
 \end{eqnarray*}
Changing the summation index $\alpha_1=k -2 -\alpha_2 - \alpha_3$, we can write
\begin{eqnarray*} \sum_{\alpha_1, \alpha_2=0,k=0}^{\infty}   \frac{1}{\nu_1^2} \frac{  ( \alpha_3 +1) \left( k + \alpha_2  +  \alpha_3 + 3 \right)   \Gamma\left( k + \alpha_2 +  2 \right)  }{  \pi^3  \Gamma\left( \alpha_2 +1 \right)  \Gamma\left( k + 1 \right) } \nu_1^{k} \left( \frac{\nu_2}{\nu_1} \right)^{\alpha_2} \left(\frac{\nu_3}{\nu_1}\right)^{\alpha_3}
 \end{eqnarray*}
After little calculation using $z\Gamma(z)=\Gamma(z+1)$, we have
\begin{eqnarray*} \sum_{\alpha_1, \alpha_2=0,k=0}^{\infty}   \frac{1}{\nu_1^2} \frac{  ( \alpha_3 +1)^2   \Gamma\left( k + \alpha_2 +  2 \right)  }{  \pi^3  \Gamma\left( \alpha_2 +1 \right)  \Gamma\left( k + 1 \right) } \nu_1^{k} \left( \frac{\nu_2}{\nu_1} \right)^{\alpha_2} \left(\frac{\nu_3}{\nu_1}\right)^{\alpha_3} \\ + \sum_{\alpha_1, \alpha_2=0,k=0}^{\infty}   \frac{1}{\nu_1^2} \frac{  ( \alpha_3 +1)   \Gamma\left( k + \alpha_2 +  3 \right)  }{  \pi^3  \Gamma\left( \alpha_2 +1 \right)  \Gamma\left( k + 1 \right) } \nu_1^{k} \left( \frac{\nu_2}{\nu_1} \right)^{\alpha_2} \left(\frac{\nu_3}{\nu_1}\right)^{\alpha_3}
 \end{eqnarray*}

Finally using well knows formulas $\sum\limits_{k=0}^{\infty }  (k+1)^2 x^k= \frac{x+1}{(1-x)^3}$, $\sum\limits_{k=0}^{\infty } (k+1) x^k=\frac{1}{(1-x)^2}$ and $F_2(r,1,1;1,1,x,y)= \frac{1}{(1-x-y)^r}$, we obtain the desired result.

Similarly we can compute Bergman kernel for $$\Omega= \left\{ (z_1,z_2,z_3) \in \mathbb{C}^{1+n+m}  \colon  \|z_3\|^{\lambda} < |z_1|^{2p} + \|z_2\|^2, \ |z_1|^{2p} + \|z_2\|^2  < |z_1|^{p} \right\},$$ for $p, \lambda > 0$.


\bibliographystyle{amsplain}

\noindent Tomasz Beberok\\
Department of Applied Mathematics\\
University of Agriculture in Krakow\\
ul. Balicka 253c, 30-198 Krakow, Poland\\
email: tbeberok@ar.krakow.pl

\end{document}